\documentclass[11pt,reqno]{amsart}
\usepackage{hyperref}
\usepackage{amssymb}
\usepackage{calrsfs}
\evensidemargin=\oddsidemargin
\newcommand\Z{{\mathbf Z}}
\newcommand\R{\mathbf R}
\newcommand\C{\mathbf C}
\newcommand\bH{{\mathbf H}}
\newcommand\bF{{\mathbf F}}

\newcommand\SL{\operatorname{\sf SL}}
\newcommand\SO{\operatorname{\sf SO}}
\newcommand\Sp{\operatorname{\sf Sp}}
\newcommand\SU{\operatorname{\sf SU}}
\renewcommand\S{\operatorname{\sf S}}
\renewcommand\O{\operatorname{\sf O}}
\newcommand\U{\operatorname{\sf U}}

\newcommand\so{{\mathfrak{so}}}
\newcommand\su{{\mathfrak{su}}}

\renewcommand\sp{{\mathfrak{sp}}}

\renewcommand\k{{\mathfrak{k}}}

\newcommand\fsl{{\mathfrak{sl}}}

\newcommand\op{{\oplus}}

\newcommand\calB{{\mathcal B}}
\renewcommand\a{\mathfrak a}
\newcommand\g{{\mathfrak{g}}}

\newcommand\m{{\mathfrak{m}}}
\newcommand\n{{\mathfrak{n}}}
\newcommand\p{{\mathfrak{p}}}
\newcommand\q{{\mathfrak{q}}}
\newcommand\s{{\mathfrak{s}}}

\newcommand\gl{{\mathfrak{gl}}}
\newcommand\ip{{\langle\,,\,\rangle}}
\newcommand\isom{{\text{Isom}}}

\newcommand\tr{\operatorname{tr}}

\newcommand\diag{\operatorname{diag}}

\newcommand\ad{\operatorname{ad}}
\newcommand\Id{\operatorname{Id}}

\newcommand\ric{\operatorname{Ric}}
\theoremstyle{plain}
\newtheorem{theorem}{Theorem}

\newtheorem{lemma}[theorem]{Lemma}
\newtheorem{proposition}[theorem]{Proposition}

\theoremstyle{definition}
\newtheorem*{definition}{Definition}
\newtheorem*{example}{Example}

\theoremstyle{remark}
\newtheorem*{remark}{Remark}
\numberwithin{equation}{section}
\numberwithin{theorem}{section}

\begin{document}
\title[New Einstein Metrics]
{New examples of non-symmetric Einstein solvmanifolds of negative Ricci curvature}

\subjclass[2000]{Primary: 53C30}

\author[M. M. Kerr]{Megan M. Kerr}
\address{Department of Mathematics, Wellesley College, 106 Central St.,
Wellesley, MA 02481} \email{mkerr@wellesley.edu}

\date{\today}

\begin{abstract}
We obtain new examples of  non-symmetric Einstein solvmanifolds by combining two techniques.  
 In \cite{T2}, H.~Tamaru  constructs 
new {\em attached} solvmanifolds, which are submanifolds of the solvmanifolds corresponding to noncompact symmetric spaces, endowed with a natural metric.  
Extending this construction,  we  apply it to {\em associated} solvmanifolds, described in \cite{GK}, %by C.~S.~Gordon and the author, 
obtained by  modifying the algebraic structure of the solvable Lie algebras  corresponding to noncompact symmetric spaces.   
Our new examples are  Einstein solvmanifolds with nilradicals of high nilpotency, which are geometrically distinct from noncompact symmetric spaces and their submanifolds. 
\end{abstract}
\maketitle

%%%%%%%%%%%%%%%%%%%%%%%%%%%%%%%%%%%%%%%%%%%%%%%%%%%%%%%%%%%%%%%%%%%%%%%%%%%%%%%%

\section{Introduction} 
In this paper we construct new examples of solvmanifolds of constant negative Ricci curvature.  These spaces provide many explicit new examples of homogeneous Einstein manifolds that are neither the solvmanifolds corresponding to noncompact symmetric spaces, nor their submanifolds.  We obtain our examples by extending the method of H.~Tamaru in \cite{T2}, in which, via parabolic subalgebras of semisimple Lie algebras, he builds solvable subalgebras  by choosing a subset $\Lambda'$ of the set $\Lambda$ of simple roots. Tamaru proves that the solvable subalgebra of the restricted root system, $\s_{\Lambda'}$, given a natural inner product, called an {\em attached} solvmanifold, is in fact an Einstein solvmanifold. 
We combine this with a method introduced by C.~S.~Gordon and the author in \cite{GK} to construct Einstein solvmanifolds which are {\em associated}  (but not isometric) to the Einstein solvable Lie groups corresponding to higher rank, irreducible symmetric spaces of noncompact type.  

Tamaru's construction yields solvmanifolds which are naturally homogeneous submanifolds of symmetric spaces of noncompact type; however,  they generally are {\em not}  totally geodesic subalgebras.  
When Tamaru's method is extended to construct  the Einstein solvmanifolds here, which are {\em attached} to {\em associated} solvmanifolds,  we get completely new examples.  We show, by isometry groups, that they cannot be submanifolds of symmetric spaces. 
Furthermore, we show that while our examples have constant negative Ricci curvature, they admit two-planes of positive sectional curvature.  

The category of solvmanifolds is a natural place to look for Einstein manifolds of negative Ricci curvature.  All known examples of homogeneous Einstein manifolds of negative Ricci curvature are Riemannian solvmanifolds. Indeed, in 1975 Alekseevskii conjectured that every noncompact homogeneous Einstein manifold is a solvmanifold  \cite{Al}.  There has been much recent progress toward determining existence of Einstein manifolds of negative Ricci curvature, and more generally, Ricci solitons (e.g. \cite{J1,J2, J3, L1, L2, L3, LW, N, P}). 

It is a pleasure to thank Carolyn Gordon for helpful discussions. 
%%%%%%%%%%%%%%%%%%%%%%%%%%%%%%%%%%%%%%%%%%%%%%%%%%%%%%%%%%%%%%%%%%%%%%%%%%%%%%%%
\section{Preliminaries} 
A solvable Lie group $S$  together with a left-invariant metric $g$ is called a {\em Riemannian solvmanifold}. A {\em metric Lie algebra} is a Lie algebra endowed with an inner product. 
The left-invariant metric $g$ on a solvmanifold $S$ defines an inner product $\ip$ on the Lie algebra $\s$ of $S$. Conversely, an inner product on a Lie algebra $\s$ defines a unique left-invariant metric on $S$, the corresponding simply connected solvable Lie group.   Given the  one-to-one correspondence between simply connected Riemannian solvmanifolds and solvable metric Lie algebras,  it makes sense to work at the Lie algebra level.  

\begin{definition} Let $(S,g)$ be a simply connected Riemannian solvmanifold, and let $(\s, \ip)$ be the corresponding solvable metric Lie algebra. Write $\s = \a \oplus \n$, where $\n$ is the nilradical of $\s$. We say the metric Lie algebra $(\s, \ip)$ and the solvmanifold $(S,g)$ are of {\em Iwasawa type} if \\
(i)  $\a$ is abelian, \\
(ii) for all $A \in\a$, $(\ad A)$ is symmetric relative to the inner product $\ip$ on $\s$, \\
(iii) if $A \neq 0$, $(\ad A) \neq 0$, and \\
(iv) for some $A_0 \in \a$, the restriction $(\ad A_0)|_{\n}$  is positive definite on $\n$. 
\end{definition}

\begin{definition} We say two metric Lie algebras are {\em isomorphic} if there exists a Lie algebra isomorphism $\varphi:\s \to \s'$ which is also an inner product space isometry.  
\end{definition}
In what follows, all solvmanifolds will be of Iwasawa type, with real roots. This guarantees our solvmanifolds are in standard position within the identity component of the isometry group of $S$ \cite{GW}.  In this case, the question of when two solvmanifolds are isometric is answered at the metric Lie algebra level. 
\begin{theorem}\cite[5.2]{GW} Let $(\s, \ip)$ and $(\s',\ip')$ be data in standard position for simply connected Riemannian  solvmanifolds $S$ and $S'$, respectively. Then $S'$ is isometric to $S$ if and only if their corresponding metric Lie algebras are isomorphic. 
\end{theorem}
We will use this  to verify that our examples are distinct from previously described solvmanifolds.
%%%%%%%%%%%%%%%%%%%%%%%%%%%%%%%%%%%%%%%%%%%%%%%%%%%%%%%%%%%%%%%%%%%%%%%%%%%%%%%%
\subsection{Symmetric spaces and subalgebras}   

Some solvmanifolds correspond to symmetric spaces of non-compact type in the following way (see Helgason \cite{Hel} for more details). Given a noncompact symmetric space $M$, let $G = \isom(M)$, the full isometry group. The isotropy subgroup $K$ at a point of $M$ is a maximal compact subgroup of $G$, so that $M \cong G/K$.  
The Lie algebra $\g$ of $G$ admits an Iwasawa decomposition $\g=\k +\a +\n$ where $\k$ is the Lie algebra of $K$, $\a$ is abelian, $\n$ is nilpotent.  We note that  the solvable Lie subgroup $S$ of $G$, with Lie algebra  $\s:=\a+\n$, acts simply transitively on $M$,  so that we may identify $M$ with the solvmanifold  $S$.

Alternatively---equivalently---we may begin with a semisimple Lie algebra $\g$ and a Cartan involution $\sigma$ on $\g$. Via the Cartan involution $\sigma$, we obtain a decomposition $\g = \k + \p$,  where $\sigma|_{\k} = \Id$,  $\sigma|_{\p} = -\Id$. Using the Killing form $B$, we define  $B_{\sigma}(X,Y) := -B(X,\sigma(Y))$, an $\ad_\k$-invariant inner product on $\g$.

Let $\a$ be a maximal torus in $\p$. We let  $\Delta$ denote a root system of $\g$, and we take the root space decomposition $\g = \g_0 + \sum_{\alpha \in \Delta} \g_{\alpha}$, where $\g_{\alpha} = \{X \in \g \mid (\ad A)X =\alpha(A) X \text{ for all } A\in\a\}$,  
$\g_0$ is the centralizer of $\a$ in $\g$.  Denote by  $\Delta^+$ the set of positive roots and take $\Lambda \subset  \Delta^+$ to be a set of simple roots.  Using $\Delta^+$, we define $\n := \sum_{\alpha \in \Delta^+} \g_{\alpha}$, the nilradical of $\g$.  Define $\s := \a + \n$, endowed with this inner product  induced from $B_{\sigma}$: $\ip = 2B_{\sigma}|_{\a \times \a} + B_{\sigma}|_{\n \times \n}$.  It is straightforward to see that the corresponding simply connected solvmanifold is a symmetric space, and 
the {\em rank} of the symmetric space is $|\Lambda |$, the number of elements in $\Lambda$ (cf. \cite[p.532]{Hel}). 

%%%%%%%%%%%%%%%%%%%%%%%%%%%%%%%%%%%%%%%%%%%%%%%%%%%%%%%%%%%%%%%%%%%%%%%%%%%%%%%%\subsection{Subalgebras} 

We are ready to describe Tamaru's metric Lie subalgebras $\s_{\Lambda'}$ \cite{T2}. We begin with a description of the subalgebra construction, then we review his proof that the appropriate metric on $\s_{\Lambda'}$ has constant Ricci curvature. 
%so that $S_{\Lambda'}$ is an Einstein solvmanifold.

We have  $\Lambda = \{\alpha_1, \dots, \alpha_r\}$, the set of  simple roots. We take a dual basis $\{H^1, \dots, H^r\}$ in $\a$, so that $\alpha_i(H^j) = \delta_{ij}$.  Choose a {\em characteristic element}, 
$$Z = c_{i_1} H^{i_1} + c_{i_2} H^{i_2} + \dots + c_{i_k} H^{i_k},$$ 
for positive integers $c_{i_1},\dots,c_{i_k}$.  Note that every eigenvalue of $\ad Z$ is an integer  \cite{KA}. For each $j \in \Z$, define $\g^j := \sum_{\alpha(Z) =j} \g_{\alpha}$, summing over  $\alpha \in \Delta$. 
 For any choice of characteristic element $Z$, the eigenspace decomposition $\g = \bigoplus \g^j$ gives a grading of $\g$: 
 $[\g^i, \g^j] \subset \g^{i+j}$ for all $i,j\in\Z$.  
Define $\Lambda' = \{ \alpha_i  \in \Lambda \mid \alpha_i(Z)=0\}$. Equivalently, $\Lambda' = \Lambda \setminus \{\alpha_{i_1}, \dots, \alpha_{i_k}\}$.
%Then $\langle \Lambda'\rangle = \{ \beta \in \Delta \mid \beta(Z) =0\}$. 

Just as $\g_0$ is the centralizer of $\a$ in $\g$, we denote by $\k_0$ the centralizer of $\a$ in $\k$.  The subalgebra $\k_0 \op \a \op \n$ is called a {\em minimal parabolic subalgebra} of $\g$. Any subalgebra $\q \subset \g$ which contains $\k_0 \op \a \op \n$ (up to conjugation)  is a {\em parabolic subalgebra}.  Every parabolic subalgebra can be constructed from a subset  $\Lambda' \subset \Lambda$  \cite[7.76]{Kn}. 
In \cite[3.6]{T2}, Tamaru proves $\q_{\Lambda'} := \sum_{k \geq 0} \g^k$ is a parabolic subalgebra of $\g$.   On $\q_{\Lambda'}$, one has the following Langlands decomposition \cite[3.8]{T2}: 
\begin{itemize}
\item $\a_{\Lambda'} = \text{span}\{H^{i_1}, \dots,H^{i_k}\}$, 
\item $\m_{\Lambda'} = \g^0 \ominus\a_{\Lambda'}$,  
\item $\n_{\Lambda'} = \sum_{k > 0} \g^k$.  
\end{itemize}
Finally, our  solvable metric Lie algebra is $\s_{\Lambda'} := \a_{\Lambda'} + \n_{\Lambda'}$, endowed with the inner product inherited from that of $\g$: 
$\langle ~,~\rangle= 2B_{\sigma}|_{\a_{\Lambda'} \times \a_{\Lambda'}} + B_{\sigma}|_{\n_{\Lambda'} \times \n_{\Lambda'}}.$
%Thus we have our metric solvable Lie algebra $(\s_{\Lambda'}, \ip)$. 
We say the corresponding simply connected solvmanifold $(S_{\Lambda'},g)$ is an {\em attached solvmanifold}.  

\begin{remark} A subset $\Lambda' \subset \Lambda$ is said to be {\em trivial} if $\Lambda'$ and $\Lambda \setminus \Lambda'$ are orthogonal. Provided $\Lambda' \subset \Lambda$  is nontrivial,
the solvmanifold $(S_{\Lambda'},g)$   will not be a totally geodesic submanifold of  $(S_{\emptyset},g)$ \cite[6.4]{T2}. \end{remark}

\begin{proposition} \cite[4.4]{T2} When  $\Lambda' $ is chosen to be the empty set, the attached solvmanifold  $(\s_{\emptyset} = \a_{\emptyset} + \n_{\emptyset},\ip)$   is a symmetric space and an  Einstein manifold.  
\end{proposition}

Now, take any nonempty $\Lambda'$. Tamaru proves that 
$\s_{\Lambda'} = \a_{\Lambda'} + \n_{\Lambda'}$ is a subalgebra of $\s_{\emptyset}$, and the Ricci curvature of $\s_{\Lambda'}$ is the same as that of  $\s_{\emptyset}$.  Thus $(\s_{\Lambda'}, \langle ~,~\rangle)$ is an Einstein solvmanifold \cite[5.3]{T2}.  We outline Tamaru's proof below. 

%%%%%%%%%%%%%%%%%%%%%%%%%%%%%%%%%%%%%%%%%%%%%%%%%%%%%%%%%%%%%%%%%%%%%%%%%%%%%%%%
\subsection{Ricci curvature} 
Recall the sectional and Ricci curvature formulae, viewed as  bilinear forms on $\s$ (cf.~\cite{Bes}).    
Let $\{ X_1, \dots, X_n\}$ be an orthonormal basis for $\s$. The Levi-Civita connection $\nabla_X Y$ of $(S,g)$ is given by 
$$\nabla_X Y = \tfrac12 [X,Y] + U(X,Y)$$ 
where $U:\s \times \s \to \s$ is the symmetric bilinear form satisfying
$$\langle U(X,Y),Z\rangle = \tfrac12 \langle [Z,X],Y\rangle + \tfrac12 \langle [Z,Y],X \rangle.$$
 Let $B$ denote the Killing form for $\s$. Then for a pair of orthonormal vectors $X,~Y \in \s$, 
 the sectional curvature of the two-plane spanned by $X$ and $Y$ is given by
\begin{equation*}%\label{sec}
 K(X,Y)= -\tfrac 34 |[X,Y]|^2 - \tfrac 12 \langle [X,[X,Y]],Y\rangle - \tfrac 12 \langle [Y,[Y,X]],X\rangle +|U(X,Y)|^2 
 - \langle U(X,X),U(Y,Y)\rangle.
 \end{equation*}
 The Ricci curvature is given by 
\begin{equation*}%\label{ric}
\ric(X,Y) = -\tfrac 12 \sum_i \langle [X, X_i], [Y, X_i]\rangle - \tfrac 12 B(X,Y) + \tfrac 14 \sum_{i,j} \langle [X_i,X_j],X\rangle \langle [X_i,X_j], Y\rangle - \langle U(X, Y),H_0 \rangle
\end{equation*}
where $H_0 = \sum U(X_i,X_i)$ is the mean curvature vector. 

Since these solvable metric Lie algebras are of Iwasawa type, the Ricci curvature expression simplifies:
\begin{lemma}\cite[1.4]{W} Let $\ric$ denote the Ricci curvature tensor on $\s = \a \op \n$, a solvable metric Lie algebra.  Let $\ric^{\n}$ denote the Ricci curvature tensor on $\n$. For $X \in \n$, $A,A' \in \a$, \\
(i) $\ric(A,A') =-\tr(\ad A \circ \ad A')$, \\
(ii) $\ric(A,X) =0$, \\
(iii) $\ric(X) = \ric^{\n}(X) - \langle [H_0,X],X\rangle$, where $H_0$ is the mean curvature vector.
\end{lemma}

The  proof that Ricci curvature is unchanged relies on Wolter's lemma and these observations. (Recall, $\g^0$ is the sum of the root spaces $\g_{\alpha}$ for which $\alpha(Z)=0$.)
\begin{enumerate}
\item Since $[\a_{\Lambda'}, \g^0] =0$, we see   
${\displaystyle \tr(\ad A \circ \ad A')|_{\n_{\Lambda'}} = \tr(\ad A \circ \ad A')|_{\n_{\emptyset}}}$ for all $A,A' \in \a_{\Lambda'}$.  
Thus $\ric^{\s_{\Lambda'}}(A,A') = \tfrac 14 \langle A, A' \rangle = \ric^{\s_{\emptyset}}(A,A')$. 
\item For any $A \in  \a_{\Lambda'}$ and any $X \in \n_{\Lambda'}$, $\ric^{\s_{\Lambda'}}(A,X) = 0 = \ric^{\s_{\emptyset}}(A,X)$. 
\item Let $\{ E_i'\}$ be an orthonormal basis for $\n_{\Lambda'}$, and let $\{E_j^\perp\}$ be an orthonormal basis for 
$\n_{\emptyset} \ominus \n_{\Lambda'}$.    Define $H_0' = \tfrac 12 \sum [\sigma E_i',E_i']$ and $H^\perp = \tfrac 12 \sum [\sigma E_j^\perp,E_j^\perp]$, so that $H_0 = H_0' + H_0^\perp$.  
Hence $\ric^{\s_{\emptyset}}(X,Y) = \langle \ric^{\n_{\emptyset}} (X),Y\rangle - \langle [H_0,X],Y\rangle$,   
$\ric^{\s_{\Lambda'}} (X,Y) = \langle \ric^{\n_{\Lambda'}} (X),Y \rangle - \langle [H_0',X],Y\rangle.$  
\item We see ${\displaystyle \ric^{\n_{\emptyset}}(X) = -\tfrac 14 \sum [E_i,[\sigma E_i,X]_{\n_{\emptyset}}] + \tfrac 12\sum[\sigma E_i, [E_i,X]]_{\n_{\emptyset}}}$,   \\ and ${\displaystyle \ric^{\n_{\Lambda'}}(X) = -\tfrac 14 \sum [E_j',[\sigma E_j',X]_{\n_{\Lambda'}}] + \tfrac 12\sum[\sigma E_j', [E_j',X]]_{\n_{\Lambda'}}.}$
\end{enumerate}

\noindent Using these, Tamaru proves  ${\displaystyle \ric^{\n_{\emptyset}} (X) - \ric^{\n_{\Lambda'}} (X) = [H_0^\perp,X]}$ for all 
$X \in \n_{\Lambda'}$.  Thus,
\begin{align}\ric^{\s_{\emptyset}}(X,Y) - \ric^{\s_{\Lambda'}} (X,Y) &= \langle \ric^{\n_{\emptyset}}(X) - \ric^{\n_{\Lambda'}}(X), Y\rangle - \langle [H_0 - H_0',X],Y\rangle \notag \\
&= \langle [H_0^\perp,X],Y\rangle - \langle [H_0^\perp,X],Y\rangle = 0.\notag \end{align}

%%%%%%%%%%%%%%%%%%%%%%%%%%%%%%%%%%%%%%%%%%%%%%%%%%%%%%%
%%%%%%%%%%%%%%%%%%%%%%%%%%%%%%%%%%%%%%%%%%%%%%%%%%%%%%%

\subsection{Associated solvmanifolds} 
We now review the method in \cite[4.2]{GK} of modifying the algebraic structure of a solvmanifold while preserving the Ricci curvature. 
 We begin with our real semisimple Lie algebra $\g$. Let $\g^\R = (\g)^\C$, the complexification of $\g$, viewed as a real Lie algebra. 
View $\g$, and also $\s$, as  subalgebras of $\g^\R$.  
 Recall that the root
spaces $\g_{\alpha}$ are mutually orthogonal relative to the inner product defined by the symmetric metric on $S$.  
\begin{definition} There exists an {\em adapted} orthonormal basis $\calB$ 
of the subspace $\n$ of $\s$ such that 
\begin{itemize}
\item each basis vector lies in some $\g_{\beta}$,
\item the bracket of any two basis vectors is a scalar multiple of another basis vector,  
\item for any $X$, $Y$, and $U$  basis vectors with $Y\neq U$, then
$[X,Y]\perp [X,U].$ \end{itemize}
\end{definition}

\begin{definition}\label{assoc} Let $\calB$ be an adapted orthonormal basis $\calB$ for $\n$. 
Consider a new subspace of $\g^\R$ with basis  $\calB'$, where we obtain 
 $\calB'$ from  $\calB$ by replacing some of the vectors $X\in \calB$ by $\sqrt{-1}\,X$.  
 Our new subspace $\n'$ will be a nilpotent algebra provided it is closed under the
bracket operation in $\g$. The brackets of the basis vectors in $\n'$ differ only by a sign from their corresponding brackets in
$\n$.  Notice the subalgebra $\a$ of $\g^\R$ normalizes $\n'$.  We choose the inner product on  $\n'$ which makes $\calB'$ orthonormal, and choose the inner product on $\s' := \a \oplus \n'$ as an orthogonal sum of inner product spaces.   Then
$\s'$ is a solvable metric Lie algebra.  
%We will say $\s$ and $\s '$ are {\it associated} subalgebras of $\g^\R$.  
%This inner product on $\s '$ defines a left-invariant Riemannian metric on the associated simply connected Lie group $S'$. 
We say the corresponding Riemannian solvmanifold $S'$ is  {\em associated} to the symmetric space $G/K$.
\end{definition}
\begin{remark} The simply connected solvmanifold $S'$ corresponding to $\s'$ is an Einstein manifold:  the sign changes are cancelled out in the Ricci curvature equation \cite[4.3]{GK}.   It has since been proved that whenever $\n$ and $\n'$ are two real nilpotent Lie algebras whose complexifications are isomorphic, if $\n$ is an Einstein nilradical then  $\n'$ is as well 
(\cite[6]{N},\cite[6.5]{J3}).
\end{remark}

We note that one can first construct an associated solvmanifold $S'$ (associated to $S$), then consider the solvmanifolds 
${S'}_{\Lambda'}$  attached to $S'$. Alternatively, one can start from solvmanifold $S$, construct the  solvmanifolds $S_{\Lambda'}$ attached to $S$, and from there, create their associated solvmanifolds ${S_{\Lambda'}}'$.  These two techniques can be combined in either order to obtain the same new Einstein solvmanifolds.
%%%%%%%%%%%%%%%%%%%%%%%%%%%%%%%%%%%%%%%%%%%%%%%%%%%%%%%%%%%%%
\subsection{Isometries of solvmanifolds} 
We come to the issue of determining whether we have indeed created new solvmanifolds.  We verify that  these examples are distinct from Tamaru's.  To do this, we note that every attached solvmanifold $S_{\Lambda'}$ will inherit isometries from 
$S_{\emptyset}$.

\begin{proposition} Consider $\isom(S_{\Lambda'})$, the group of isometries of our solvmanifold. Let $M_{\Lambda'}$ denote the compact component of  $\isom(S_{\Lambda'})$.  Then $M_{\emptyset} \subseteq M_{\Lambda'}$.
\end{proposition}
\begin{proof} We know from the work of Gordon and Wilson that the compact component $M$ of the isometry group of a solvmanifold $S$ is the normalizer of $A$ in $K$, the maximal compact subgroup  of $G=\isom(S)$ \cite{GW}. That is,  $M_{\emptyset}$ is the normalizer of $A$ in $K$ (here $S_{\emptyset}=G/K$, and $G=\isom(S_{\emptyset}) =KAN$). The isometries of 
a solvmanifold $S$ must preserve root spaces; hence all isometries of $S_{\emptyset}$ take $S_{\Lambda'}$ to itself. Our conclusion,  $M_{\emptyset} \subseteq M_{\Lambda'}$, follows directly \cite[La.~6.4]{GW}.  \end{proof}
  
Now, we contrast the attached solvmanifold case with the case of associated solvmanifolds. 
  
\begin{proposition}\cite[4.4]{GK}\label{isomorphic} If $S'$ and $S''$ are two Riemannian solvmanifolds associated
with $G/K$, then $S'$ is isometric to $S''$ if and only if there exists an automorphism
$\phi$ of the root system (i.e., a Weyl group element) and an isomorphism
$\tau: \n' \to \n''$ such that
$\tau(\n'_\alpha)=\n''_{\phi(\alpha)}$ for all roots $\alpha$.
 \end{proposition} 

%\begin{corollary} Let $S_1 := (S')_{\Lambda_1'}$ be attached to solvmanifold $S'$, associated to $S$, and 
%let $S_2 := (S'')_{\Lambda_2'}$ be attached to  solvmanifold $S''$, associated to $S$. Assume each $\Lambda_i'$ is nontrivial. 
%Then $S_1$  is isomorphic (and isometric) to $S_2$  only if  there exists an automorphism $\phi$ of root systems 
%taking $\Lambda_1'$ to $\Lambda_2'$. 
%\end{corollary}
%\begin{proof} We know $S_1$ and $S_2$ will be isometric if and only if they are isomorphic, if and only if their corresponding %Lie algebras are isomorphic. If $\s_1 \cong \s_2$ then $\s_1^{\C} \cong\s_2^{\C}$. Thus, we need $\Lambda_1'$ and 
%$\Lambda_2'$ to be either equal or equivalent via a root system automorphism. 
%\end{proof}

\begin{proposition} Let $S'$ be a solvmanifold associated to the symmetric space $G/K$.  Consider $\isom(S')$, the group of isometries of our solvmanifold $S'$. Let $M$ denote the compact component of  $\isom(G/K)$ and $M'$ denote the compact component of  $\isom(S')$.  Then $M' \subset M$.
\end{proposition}
\begin{proof} Recall, $M \subset \isom(G/K)$ is the normalizer of $A$ in $K$. Since the isometries in 
$M'$ must preserve root spaces,  we see that  when a selection of  basis  vectors $X$ are replaced by $\sqrt{-1}\,X$ to construct $\calB'$ (and hence $S'$), some isometries may be lost.  \end{proof}

In what follows, in each case we will verify that  for our associated solvmanifolds $S'$, $M' \subsetneq M$. That is, the compact component $M'$ of  $\isom(S')$ is a proper subset of $M$.

%%%%%%%%%%%%%%%%%%%%%%%%%%%%%%%%%%%%%%%%%%%%%%%%%%%%%%%%%%%%%%%%%%%%%%%%%%%%%%%%
\section{New Examples} 
\subsection{\boldmath $SO(p,q)/SO(p)SO(q)$, \boldmath $SU(p,q)/S(U(p)U(q))$, and \boldmath $Sp(p,q)/Sp(p)Sp(q)$}\label{SO(p,q)} 
We start with the  noncompact symmetric spaces of the form $\SO(p,q)/\S(\O(p)\O(q))$, $\SU(p,q)/\S(\U(p)\U(q))$ or 
$\Sp(p,q)/\Sp(p)\Sp(q)$. Without loss of generality, we may assume that $q\geq p$, so that the rank of these symmetric spaces is $p$.  Then
$$\g = \{ A \in \gl(p+q, \bF) \mid AM = -MA^t \} \quad \text{where} \quad M = \begin{pmatrix} \Id_p & 0 \\ 0 & -\Id_q \end{pmatrix},$$
where $\bF = \R$ in the case of $\so(p,q)$, $\bF=\C$ in the case of $\su(p,q)$, and $\bF=\bH$ in the case of $\sp(p,q)$.
Elements of $\g$  may be written in matrix block form this way,
${\displaystyle \begin{pmatrix} A & B & C \\ B^t & D & E \\ C^t & -E^t & F \end{pmatrix}, }$
such that blocks $A,B,D$ are $p \times p$,  blocks $C, E$ are $p \times (q-p)$, and block $F$ is $(q-p)\times(q-p)$, and furthermore,  $A=-A^t$, $D=-D^t$, $F=-F^t$, while $B, C, E$ are arbitrary.  

We consider the maximal torus $\a$ consisting of the matrices in $\g$ for which $B$ is diagonal, and all other blocks are zero.  
Let $\{ \omega_1, \dots,\omega_p\}$ denote the basis of duals to the standard basis for the set of diagonal elements of $B$. Then we may choose the set of positive roots
\begin{align}\Delta^+ &= \{ \omega_i \mid 1 \le i \le p\} \cup \{\omega_j \pm \omega_i \mid 1 \le i < j \le p\} \text{ (orthogonal)} \notag \\  
\Delta^+ &= \{ \omega_i \mid 1 \le i \le p\} \cup \{\omega_j \pm \omega_i \mid 1 \le i < j \le p\} \cup \{2\omega_i \mid 1 \le i \le p\} \text{ (unitary, symplectic).}\notag \end{align}   
We may take 
${\displaystyle \Lambda = \{ \omega_1, \omega_2 - \omega_1, \omega_3 - \omega_2 \dots, \omega_p - \omega_{p-1}\}}$ as the fundamental system.  We give bases of  root vectors below. In what follows, $E_{ij}$ is the skew-symmetric matrix with 1 in the 
$(i,j)$ entry, $-1$ in the $(j,i)$ entry and zeroes elsewhere; $F_{ij}$ is the symmetric matrix with 1 in the $(i,j)$ and $(j,i)$ entries and zeroes elsewhere.  (While the basis elements for $\Lambda$ above are not orthogonal, the basis vectors below are orthonormal.) 
\begin{align} \n_{\omega_k} = \text{span}\{ U_{km} &= (F_{km} + E_{p+k,m}), \, \, \epsilon \, V_{k,m} =\epsilon(E_{km} + F_{p+k,m}) \mid 2p+1 \leq m \leq p+q\} \notag \\
\n_{\omega_j - \omega_i} =  \text{span}\{ Y_{jk}^-  &= \tfrac{1}{\sqrt 2} (E_{jk} +  E_{p+j,p+k} -F_{j,p+k} - F_{k,p+j}), \notag \\
\epsilon \, Z_{jk}^- &= \tfrac{\epsilon}{\sqrt 2} ( F_{jk} + F_{p+j,p+k} - E_{j,p+k}  + E_{k,p+j}) \} \notag \\
\n_{\omega_j + \omega_i} =  \text{span}\{ Y_{jk}^+  &= \tfrac{1}{\sqrt 2} (E_{jk} -  E_{p+j,p+k} -F_{j,p+k} + F_{k,p+j}), \notag \\
\epsilon \, Z_{jk}^+ &= \tfrac{\epsilon}{\sqrt 2} ( F_{jk} - F_{p+j,p+k} - E_{j,p+k}  - E_{k,p+j}) \} \notag \\
\n_{2\omega_k} = \text{span}\{\epsilon \, W_k &= \epsilon (F_{kk} - F_{p+k,p+k} - E_{k,p+k}) \}
\notag \end{align}
The basis vectors of type $V_{ij}$, $Z_{ij}^{\pm}$ and $W_j$ occur only in the unitary and symplectic cases, and $\epsilon^2 = -1$. In the unitary case, $\epsilon = i$. In the symplectic case, $\epsilon \in \{i,j,k\}$.

Notice that for each $k$, the root space $\n_{\omega_k}$  lies in  
${\displaystyle W = \left\{\begin{pmatrix} 0 & 0 & C \\ 0 & 0 & C \\ C^t & -C^t & 0 \end{pmatrix} \right\} }$. In particular, 
$\n_{\omega_k}$ is the space for which all but the $k^{th}$ row of $C$ is zero.  Meanwhile, the root spaces 
$\n_{\omega_j - \omega_i}$, 
$\n_{\omega_j + \omega_i}$, and $\n_{2\omega_i}$ lie in 
${\displaystyle \begin{pmatrix} * & * & 0 \\ * & * & 0 \\ 0 & 0 & 0 \end{pmatrix} }$.  These observations will be useful in what follows. 

In the case that our space has $q-p \geq 2$, we obtain an {\em associated} solvmanifold as follows. 
For any integer $a$ with $1 \le a \le q-p$ and for $b$ such that
$a+b=q-p$, we obtain a decomposition $W = W_a \oplus W_b$. In $W_a$, the last $b$ columns of $C$ are all zero. In $W_b$, the first $a$ columns of $C$ are all zero. 

We construct $\s_a$ by replacing each $X$ in $W_b$ by $X' = \sqrt{-1}\,X$ in $\s^{\C}$, leaving $W_a$ unchanged. 
With the following observations it is clear that $\s_a$ is closed under the bracket in $\s^{\C}$: \\
(i) $W$ commutes with all root spaces of type $\n_{2\omega_i}$, $\n_{\omega_j + \omega_i}$ \\
(ii) $W_b$ commutes with $W_a$ \\
(iii) The maximal torus $\a$ and the root spaces $\n_{\omega_j - \omega_i}$ normalize each of $W_a$ and $W_b$. 

The associated simply connected solvmanifold $S_a$ corresponding to $\s_a$ is Einstein.   
Yet  $S_a$, unlike the symmetric manifold, admits some positive sectional curvature. Thus $S_a$ cannot be isometric to its corresponding symmetric associate. 

To see a two-plane in $S_a$ admitting positive curvature, let $i$ and $j$ be integers with $1 \le i \le a < j \le a+b$. For 
$1 \leq k < l \leq p$, let $X_k$ be the element of $\n_{\omega_k}$  for which $C$ has $(k,i)$ and $(k,j)$ 
entries equal to $\tfrac{1}{\sqrt 2}$ and all other entries zero, and let $X_l$ in $\n_{\omega_l}$ be defined analogously. 
Let $X'_k$ and $X'_l$ be the corresponding unit vectors in $\s_a$. Then we have $[X'_k, X'_l]=0$ and 
$\langle U(X'_k,X'_k), U(X'_l,X'_l)\rangle =0$. On the other hand,  
$U(X'_k,X'_l)$ is a nonzero element of $\n_{\omega_l -\omega_k}$. 
Thus %by equation (\ref{sec}), 
$K(X'_k,X'_l) > 0$ \cite{GK}.

\begin{remark} Two solvmanifolds $S_a$ and $S_{a'}$ both associated to $S$ will be isometric to each other if and only if 
either $a=a'$ or $a=(q-p)-a'$, by Prop. \ref{isomorphic}. 
\end{remark}

In the construction of our solvable Lie group $\s_a$ out of $\s$, the maximal torus $\a$, the root system $\Delta$ and the fundamental system $\Lambda$  are unchanged.  Thus we have a dual basis for $\a$, $\{H^1, H^2, \dots, H^p\}$, such that $H^i$ is dual to the $i^{th}$ element  of $\Lambda$. 
Each choice of characteristic element $Z = c_{i_1} H^{i_1} + \dots + c_{i_k} H^{i_k}$ 
(with $c_{i_j} \in \Z_+$), determines $\Lambda' \subset \Lambda$ and, in turn,  $\Lambda'$ determines a solvable 
metric Lie subalgebra $(\s_a)_{\Lambda'} = \a_{\Lambda'} + (\n_a)_{\Lambda'} \subset \s_a$.  This proves the statement below.

\begin{proposition} Let $S_a$ be a solvmanifold {\sl associated to} a solvmanifold corresponding to a noncompact symmetric space of type $\SO(p,q)/\S(\O(p)\O(q))$, $\SU(p,q)/\S(\U(p)\U(q))$ or $\Sp(p,q)/\Sp(p)\Sp(q)$, in the sense of \cite{GK}. Let 
$(S_a)_{\Lambda'}$ be an attached solvmanifold to $S_a$.% in the sense of \cite{T2}. 
Then $((S_a)_{\Lambda'}, g')$ is an Einstein solvmanifold. 
\end{proposition}

For each such symmetric space, there are $\lfloor \frac{q-p}{2}\rfloor$ associated Einstein solvmanifolds (without loss of generality, $q\geq p$). For each associated solvmanifold, there are approximately $2^p - 1$ attached solvmanifolds (depending on the size of the Weyl group). 

\begin{remark}  The compact subgroup of isometries of $G/K$ is $M=\SO(q-p)$, which corresponds exactly to the block labeled $F$ in our matrix diagram.  The subspace $W$ is invariant under the elements of $M$. For the associated solvmanifold $S_a$, the compact subgroup of isometries is   $M_a = \SO(a)\SO(b) \subset \SO(q-p)$, elements for which now each of the separate components $W_a$ and $W_b$ are invariant (if $a=b$, $M_a$ also has a $\Z/2\Z$ factor, interchanging $W_a$ and $W_b$).  

Analogous statements hold for the unitary and symplectic symmetric spaces.  
\end{remark}

\begin{example}  We will consider the specific example of $S \cong S_{\emptyset}$, the solvmanifold which corresponds to the symmetric space $\SO(3,5)/\SO(3)\SO(5)$. From there we construct an associated solvmanifold $S_1$, still Einstein but no longer symmetric; for example, our solvmanifold is no longer negatively curved. Finally, we will choose a nontrivial subset 
$\Lambda' \subset \Lambda$ of fundamental roots, to study an attached  submanifold $(S_1)_{\Lambda'} \subset S_1$. 

On the Lie algebra level, we start with $\g = \so(3,5)$, $\k=\so(3)\so(5)$, and $\}$. As described above, $\a$  is made up of the matrices in $\g$ for which $B$ is diagonal and all other blocks are zero. The basic roots are given by $\Lambda = \{\omega_1, \omega_2 - \omega_1, \omega_3 - \omega_2 \}$. The set of positive roots is
$$\Delta^+ = \Lambda \cup \{\omega_2, \omega_3,  \omega_2 + \omega_1, \omega_3 + \omega_2, \omega_3 \pm \omega_1\}.$$
The vectors listed here give an adapted basis for $S$:
\begin{align} \a &= \text{span}\{F_{14}, F_{25}, F_{36}\}, \notag \\
 \n_{\omega_k} &= \text{span}\{U_{km} = F_{km}+ E_{3+k,m} \mid m=7,8\}, \quad k=1,2,3, \notag \\
\n_{\omega_j \pm \omega_i} &= \text{span}\{Y_{ij}^{\pm} = \tfrac{1}{\sqrt 2}(E_{ij} - F_{i,3+j} \pm F_{j,3+i} \mp E_{3+i,3+j})\}, \quad 1\leq i < j \leq 3. \notag  
\end{align}

In this example, there is only one non-isomorphic associated solvmanifold, $S_1$: take $a=b=1$ in the description above.  Then 
$W_a = \text{span}\{U_{17}, U_{27}, U_{37}\}$ and $W_b = \text{span}\{U_{18}, U_{28}, U_{38}\}$.   To construct our associated solvmanifold $S_1$, we replace $U_{i7}$ with $\sqrt{-1}U_{i7}$, for $i=1,2,3.$  Here we note that $S_1$ is different from $S$; we show we have introduced positive sectional curvature.  Let $X = \tfrac{1}{\sqrt 2}(U_{27} + U_{28})$ and $Y= \tfrac{1}{\sqrt 2}(U_{37} + U_{38})$. Then $[X,Y] = 0$.   Meanwhile, $U(X,Y)$ is a nonzero element of $\n_{\omega_3 - \omega_2}$, whereas  
$\langle U(X,X),U(Y,Y)\rangle =0$. Thus we get $K(X,Y) = |U(X,Y)|^2 > 0$. 

Next, we consider a characteristic element $Z$  in $\a$: we take $Z= H^1 + H^2 = F_{14} + 2F_{25} + 2F_{36}$. With this choice of $Z$, we find that  
$\Lambda' = \{\alpha \in \Lambda \mid \alpha(Z) =0\} = \{\omega_3 - \omega_2\}$. Now $\n_{\omega_3 - \omega_2}$ is removed.  We consider the same section as above,  $X = \tfrac{1}{\sqrt 2}(U_{27} + U_{28})$ and $Y= \tfrac{1}{\sqrt 2}(U_{37} + U_{38})$. One sees that  we still have $[X,Y]=0$, but now $U_{\Lambda'}(X,Y) = 0$. And because we must project from 
$\a_{\emptyset}$ to $\a_{\Lambda'}$,   we see $\langle U_{\Lambda'}(X,X),U_{\Lambda'}(Y,Y)\rangle > 0$.  Thus the section which has positive sectional curvature in $S_1$ has negative sectional curvature in $(S_1)_{\Lambda'}$. 
\end{example}

\subsection{\boldmath $SO(n,\bH)/U(n)$}\label{unitary}  We now consider the noncompact symmetric space $\SO(n,\bH)/\U(n)$, 
of rank $m =\lfloor \tfrac n2 \rfloor$. We have 
$$\g = \so(n, \bH) = \left\{ \begin{pmatrix} X & Y \\ -\overline{Y} & \overline{X} \end{pmatrix} \in \gl(2n, \C) \mid 
X = -X^t, \, Y = \overline{Y}^t \right\}.  $$
Within $\g$ we consider the maximal torus $\a = \{ H_j = \tfrac{i}{\sqrt 2} (E_{2j-1,2j} - E_{n+2j -1, n+2j}) \mid j=1,\dots, m\}$. 
We choose the set of positive roots 
\begin{align}\Delta^+ &= \{ \omega_j \pm \omega_k \mid 1 \leq j < k \leq m\} \cup \{2\omega_j \mid j=1, \dots,m\} \quad \text{(if $n$ even)} \notag \\
\Delta^+ &= \{ \omega_j \pm \omega_k \mid 1 \leq j < k \leq m\} \cup \{2\omega_j \mid j=1, \dots,m\} \cup \{\omega_j \mid j=1, \dots,m\}
\quad \text{(if $n$ odd)}. \notag 
\end{align}
For completeness, we describe the root spaces below (see \cite[4.2]{GK}):  
{\allowdisplaybreaks
\begin{align}
\n_{\omega_j \pm \omega_k} = \text{span}\{
A^{\pm}_{jk} = &\tfrac 12((E_{2j-1,2k-1} \mp E_{2j,2k} + E_{n+2j-1,n+2k-1}
\mp E_{n+2j,n+2k}) \notag
\\  +&i(\mp E_{2j-1,2k} - E_{2j,2k-1} \pm E_{n+2j-1,n+2k} +
E_{n+2j,n+2k-1})), \notag \\
B^{\pm}_{jk} = &\tfrac 12((E_{2j-1,2k} \pm E_{2j,2k-1} + E_{n+2j-1,n+2k} \pm
E_{n+2j,n+2k-1}) \notag
\\ +&i(\pm E_{2j-1,2k-1} - E_{2j,2k} \mp E_{n+2j-1,n+2k-1} +
E_{n+2j,n+2k})), \notag \\
C^{\pm}_{jk} = &\tfrac 12((E_{2j-1,n+2k} \mp E_{2j,n+2k-1} \mp E_{2k-1,n+2j}
+ E_{2k,n+2j-1}) \notag
\\ +&i(\mp E_{2j-1,n+2k-1} - E_{2j,n+2k} \pm E_{2k-1,n+2j-1} +
E_{2k,n+2j})), \notag \\
D^{\pm}_{jk} = &\tfrac 12((E_{2j-1,n+2k-1} + E_{2k-1,n+2j-1} \pm E_{2j,n+2k}
\pm E_{2k,n+2j}) \notag
\\ +&i(\pm E_{2j-1,n+2k} - E_{2j,n+2k-1} + E_{2k-1,n+2j} \mp
E_{2k,n+2j-1}))\}; \notag \\
\n_{2\omega_k} = \text{span}\{G_k = &\tfrac{1}{\sqrt 2}((E_{2k-1,n+2k-1} +
E_{2k,n+2k}) + i(E_{2k-1,n+2k} - E_{2k,n+2k-1}))\}; \notag \\
\text{(and if $n$ odd)} \quad  \n_{\omega_k} =  \text{span}\{X_k = &\tfrac{1}{\sqrt 2} 
((E_{2k,n} + E_{n+2k,2n}) +i(E_{2k-1,n} - E_{n+2k-1,2n})), \notag \\
Y_k = &\tfrac{1}{\sqrt 2}((E_{2k-1,n} + E_{n+2k-1,2n}) - i(E_{2k,n}
- E_{n+2k,2n})), \notag \\
Z_k = &\tfrac{1}{\sqrt 2}((E_{2k,2n} + E_{n,n+2k}) + i(E_{2k-1,2n}
-E_{n,n+2k-1})), \notag \\
W_k = &\tfrac{1}{\sqrt 2}((E_{2k-1,2n} + E_{n,n+2k-1}) - i(E_{2k,2n}
-E_{n,n+2k}))\}.\notag 
\end{align} }
In the case that $n$ is even, we can modify $\s$ by replacing 
each of the basis elements 
$B^-_{jk}$, $C^-_{jk}$, $A^+_{jk}$, $D^+_{jk}$ and $G_k$ 
by its respective product with $\sqrt{-1}$, an element of $\s^{\C}$, the
complexified Lie algebra. In the case that $n$ is odd, instead we modify $\s$ by replacing 
each of the basis elements
$X_k$, $Z_k$, $B^+_{jk}$, $C^+_{jk}$, $B^-_{jk}$, $C^-_{jk}$
by its respective product with $\sqrt{-1}$. 
 In either case, with this changed basis, we obtain a new Lie algebra $\s'$ 
associated to $\s$, and the corresponding simply-connected Riemannian solvmanifold
$S'$ is Einstein. 

 As with the previous examples, these new solvmanifolds are not isomorphic to their associated symmetric solvmanifolds. To see this, we show here that $S'$ admits some positive sectional curvature \cite{GK}. 
To construct a two-plane of positive curvature, we set 
$$X =\frac{1}{\sqrt 2}((A^-_{j,j+1})' + (B^-_{j,j+1})') \quad \text{and} \quad   
Y =\frac{1}{\sqrt 2} ((C^-_{j+1,j+2})' +(D^-_{j+1,j+2})').$$ 
Then one sees that  while $[X,Y] = 0$ and
$U(X,Y) = 0$, we have $U(X,X) = \frac{1}{\sqrt 2} (H_j - H_{j+1})$ and 
$U(Y,Y) =\frac{1}{\sqrt 2} (H_{j+1} - H_{j+2})$.  Thus
 $$ K(X,Y) = -\langle U(X,X),U(Y,Y) \rangle=-\langle \tfrac{1}{\sqrt 2} (H_j - H_{j+1}), \tfrac{1}{\sqrt 2} (H_{j+1} - H_{j+2})\rangle = \tfrac 12 \|H_{j+1}\| > 0.$$
 
 \begin{proposition} Let $S'$ be the solvmanifold {\sl associated to} a solvmanifold corresponding to a noncompact symmetric space of type $\SO(n,\bH)/\U(n)$, in the sense of \cite{GK}. Let $m =\lfloor \tfrac n2 \rfloor$.  Let 
${S'}_{\Lambda'}$ be an attached solvmanifold to $S'$ in the sense of \cite{T2}. 
Then $({S'}_{\Lambda'}, g')$ is an Einstein solvmanifold. 
 \end{proposition}

\begin{proof} In the construction of our solvable Lie group $\s'$, the maximal torus $\a$, the root system $\Delta$, and the fundamental system $\Lambda$  are unchanged from $\s$.  We have a dual basis for $\a$, $\{H^1, H^2, \dots, H^m\}$ such that $H^i$ is dual to the $i^{th}$ element  of $\Lambda$. Each choice of characteristic element $Z = c_{i_1} H^{i_1} + \dots + c_{i_k} H^{i_k}$  (each $c_{i_j} \in \Z_+$), determines $\Lambda' \subset \Lambda$ and, in turn,  $\Lambda'$ determines an attached solvable metric Lie subalgebra $\s'_{\Lambda'} = \a_{\Lambda'} + \n'_{\Lambda'} \subset \s'$. 
\end{proof}
For each $n$, there is one associated Einstein solvmanifold, and for each associated solvmanifold, 
there are approximately $2^{m} - 1$ attached solvmanifolds (depending on the number of symmetries of the Dynkin diagram). 

 \begin{remark} The compact subgroup of isometries of $G/K$ is $M=\SO(4)$ (where the action permutes $A^{\pm}_{jk}$, $B^{\pm}_{jk}$, $C^{\pm}_{jk}$ and $D^{\pm}_{jk}$, preserving the plus or minus type). For the associated solvmanifold $S'$, the analogous compact subgroup of isometries is   $M'=\SO(2)\SO(2)$. 
 \end{remark}
%
%
%
%
%%%%%%%%%%%%%%%%%%%%%%%%%%%%%%%%%%%%%%%%%%%%%%%
\subsection{\boldmath $SL(n,\bH)/Sp(n)$}\label{symplectic}  Finally, we consider the noncompact symmetric space $\SL(n,\bH)/\Sp(n)$, of rank $n-1$. Here 
$$\g = \fsl(n,\bH) = \left\{ \begin{pmatrix} X & -\bar{Y} \\ Y & \bar{X} \end{pmatrix} \mid X, Y \in \gl(n,\C), \tr(X + \bar{X}) = 0\right \}.$$
Within $\g$ we have $\a = \{\diag(a_1,\dots,a_n, a_1, \dots, a_n) \in \gl(2n,\R) \mid \sum_{i=1}^n a_i = 0 \}.$
For an arbitrary element $H = \diag(a_1,\dots,a_n, a_1, \dots, a_n) \in\a$, define $\omega_j \in \a^*$ by $\omega_j(H) = a_j$. The positive roots are $\Delta^+ = \{\omega_k - \omega_j \mid 1 \leq j < k \leq n\}$. 
 The root spaces are described below (see \cite[4.3]{GK}): 
{\allowdisplaybreaks
\begin{align}
\n_{\omega_j -\omega_k} = \text{span}\{&A_{jk} = i\sqrt 2\,(F_{jk}-F_{n+j,n+k}),
~B_{jk} = i\sqrt 2\,(F_{j,n+k}+F_{n+j,k}), \notag \\
&C_{jk} = \sqrt 2\,(F_{j,n+k} - F_{n+j,k}), ~ D_{jk} = \sqrt 2\,(F_{jk} +F_{n+j,n+k})\}. \notag
\end{align} }
We modify $\s$ by replacing each $A_{jk}$ and $C_{jk}$ by their respective products with $\sqrt{-1}$, to get new elements of 
$\s^{\C}$  for every $1 \leq j < k \leq n$. With these changes, we obtain a new Lie algebra $\s'$ 
associated to $\s$, and the corresponding simply-connected Riemannian solvmanifold $S'$ is Einstein. 

We show that, since $S'$  admits some positive sectional curvature, it cannot be isomorphic to $S$, the symmetric solvmanifold \cite{GK}.
Let $X =\frac{1}{\sqrt 2}(A_{ij}' + B_{ij}')$ and 
$Y = \frac{1}{\sqrt 2}(C_{jk}' + D_{jk}')$.  Then we see 
$[X,Y] = 0$ and $U(X,Y) = 0$, while $U(X,X) = H_{ij}$ and $U(Y,Y) = H_{jk}$  
(here $H_{ij} = X_{ii}-X_{jj}+ X_{n+i,n+i}-X_{n+j,n+j}$ in $\a$). Thus 
$$K(X,Y) = -\langle U(X,X),U(Y,Y) \rangle = -\langle H_{ij}, H_{jk}\rangle >0.$$ 

%The following result is analogous to the previous propositions. 

 \begin{proposition} Let $S'$ be the solvmanifold {\sl associated to} a solvmanifold corresponding to a noncompact symmetric space of type $\SL(n,\bH)/\Sp(n)$, in the sense of \cite{GK}. Let 
${S'}_{\Lambda'}$ be an attached solvmanifold to $S'$ in the sense of \cite{T2}. 
Then $({S'}_{\Lambda'}, g')$ is an Einstein solvmanifold. 
 \end{proposition}
For each such symmetric space, there is one associated Einstein solvmanifold, and for each associated solvmanifold, 
there are approximately $2^{n-1} - 1$ attached solvmanifolds (depending on the number of symmetries of the Dynkin diagram). 

 \begin{remark} As in the previous family of examples, the compact subgroup of isometries of $G/K$ is $M=SO(4)$ (where the action permutes $A_{jk}$, $B_{jk}$, $C_{jk}$ and $D_{jk}$). For the associated solvmanifold $S'$, the analogous compact subgroup of isometries  is   $M'=\SO(2)\SO(2)$. 
 \end{remark}
%%%%%%%%%%%%%%%%%%%%%%%%%%%%%%%%%%%%%%%%%%%%%%%%%%%%%%%%%%%%%%%%%%%%%%%%%%%%%%%%

\end{document}